\documentclass[12pt,reqno]{amsart} 
\usepackage{amssymb}
\usepackage{enumerate}
\usepackage[T1]{fontenc}
\usepackage[all]{xy}
\usepackage{mathrsfs}
\usepackage[colorlinks=true, citecolor=blue, linkcolor=red]{hyperref}
\usepackage[usenames]{color}

\definecolor{darkgreen}{rgb}{0,0.7,0}
\definecolor{darkred}{rgb}{0.9,0,0}
\definecolor{darkblue}{rgb}{0,0,0.9}

\newlength{\shorter}
\newcommand{\lie}[3]{\def\test{#2}\def\tst{G}\ifx\test\tst{{}^{#1}#2_{#3}}
\else{{}^{#1}\!#2_{#3}}\fi}

\newcommand{\4}[1]{\widebar{#1}}
\newcommand{\5}[1]{\widehat{#1}}
\newcommand{\9}[1]{{}^{#1}\!}   %% left conjugation

\let\oldcirc=\circ
\renewcommand{\circ}{\mathchoice
    {\mathbin{\scriptstyle\oldcirc}}{\mathbin{\scriptstyle\oldcirc}}
    {\mathbin{\scriptscriptstyle\oldcirc}}
    {\mathbin{\scriptscriptstyle\oldcirc}}}

\newlength{\upto}\newlength{\dnto}
\SelectTips{cm}{10} \UseTips   %% use computer modern tips in xy

\mathcode`\:="003A
\mathchardef\cdot="0201

\def\beq#1\eeq{\begin{equation*}#1\end{equation*}}
\def\beqq#1\eeqq{\begin{equation}#1\end{equation}}

\renewcommand{\:}{\colon}   %% as in f:X-->Y

\newcommand{\widebar}[1]{\overset{\mskip2mu\hrulefill\mskip2mu}{#1}
		\vphantom{#1}}

\newcommand{\longline}{\bigskip\centerline{\hbox to 5cm{\hrulefill}}\bigskip}

\newcommand{\mxfourb}[8]{#1&#2&#3&#4\\#5&#6&#7&#8\end{smallmatrix}\right)}
\newcommand{\mxfoura}[8]{\left(\begin{smallmatrix}#1&#2&#3&#4\\#5&#6&#7&#8\\}
\newcommand{\Mxfourb}[8]{#1&#2&#3&#4\\#5&#6&#7&#8\end{pmatrix}}
\newcommand{\Mxfoura}[8]{\begin{pmatrix}#1&#2&#3&#4\\#5&#6&#7&#8\\}
\def\trp[#1,#2,#3]{[\hskip-1.5pt[#1,#2,#3]\hskip-1.5pt]}

\DeclareMathAlphabet\EuR{U}{eur}{m}{n}
\SetMathAlphabet\EuR{bold}{U}{eur}{b}{n}

\newcommand{\higherlim}[2]{\displaystyle\setbox1=\hbox{\rm lim}
	\setbox2=\hbox to \wd1{\leftarrowfill} \ht2=0pt \dp2=-1pt
	\setbox3=\hbox{$\scriptstyle{#1}$}
	\def\test{#1}\ifx\test\empty
	\mathop{\mathop{\vtop{\baselineskip=5pt\box1\box2}}}\nolimits^{#2}
	\else
	\ifdim\wd1<\wd3
	\mathop{\hphantom{^{#2}}\vtop{\baselineskip=5pt\box1\box2}^{#2}}_{#1}
	\else
	\mathop{\mathop{\vtop{\baselineskip=5pt\box1\box2}}_{#1}}%
	\nolimits^{#2}
	\fi\fi}
\newcommand{\higherlimm}[2]{\setbox1=\hbox{\rm lim}
	\setbox2=\hbox to \wd1{\leftarrowfill} \ht2=0pt \dp2=-1pt
	\mathop{\mathop{\vtop{\baselineskip=5pt\box1\box2}}}\limits_{#1}
	\nolimits^{#2}}

\newcounter{let} 
\loop\stepcounter{let}
\expandafter\edef\csname cal\alph{let}\endcsname%
{\noexpand\mathcal{\Alph{let}}}
\ifnum\thelet<26\repeat

\newcommand{\tdef}[2][]{\expandafter\newcommand\csname#2\endcsname%
{#1\textup{#2}}}
\tdef{Iso}   \tdef{Aut}    \tdef{Out}    \tdef{Inn}    \tdef{Hom}
\tdef{End}   \tdef{Inj}    \tdef{map}    \tdef{Ker}    \tdef{Ob}
\tdef{Mor}   \tdef{Res}    \tdef{Spin}   \tdef{Id}     \tdef{Fr}
\tdef{rk}    \tdef{conj}   \tdef{incl}   \tdef{proj}   \tdef{diag} 
\tdef{trf}   \tdef{Sol}    \tdef{He}     \tdef{cj}     \tdef{Irr}
\tdef{free}  \tdef{supp}   \tdef{Ind}    \tdef{Ly}     \tdef{HS}
\tdef{Syl}
\tdef[_]{typ}   \tdef[^]{op}    \tdef[^]{ab}  \tdef{Tr}

\newcommand{\fdef}[1]{\expandafter\newcommand\csname#1\endcsname%
{\mathfrak{#1}}}
\fdef{X}  \fdef{red}  \fdef{foc}  \fdef{hyp}

\newcommand{\bbdef}[1]{\expandafter\newcommand% 
\csname#1\endcsname{\mathbb{#1}}}
\bbdef{C} \bbdef{F} \bbdef{R} \bbdef{Z} \bbdef{G} \bbdef{N}  \bbdef{B}

\newcommand{\itdef}[1]{\expandafter\newcommand\csname#1\endcsname%
{\textit{#1}}}
\itdef{PSL}  \itdef{PSU}  \itdef{SL}  \itdef{SU}  \itdef{GL} \itdef{GU}
\itdef{Sp}   \itdef{PSp}  \itdef{SO}  \itdef{SD}  \itdef{UT} \itdef{PGL}
\itdef{Sz}

\newcommand{\gen}[1]{\langle{#1}\rangle}

\let\nsg=\normal

\newcommand{\syl}[2]{\textup{Syl}_{#1}(#2)}
\newcommand{\sylp}[1]{\syl{p}{#1}}
\newcommand{\autf}{\Aut_{\calf}}

\newcommand{\outf}{\Out_{\calf}}

\newcommand{\homf}{\Hom_{\calf}}

\newcommand{\sminus}{\smallsetminus}
\newcommand{\defeq}{\overset{\textup{def}}{=}}

\let\too=\longrightarrow

\newcommand{\longleft}[1]{\;{\leftarrow%
\count255=0 \loop \mathrel{\mkern-6mu}%
    \relbar\advance\count255 by1\ifnum\count255<#1\repeat}\;}
\newcommand{\longright}[1]{\;{\count255=0 \loop \relbar\mathrel{\mkern-6mu}%
    \advance\count255 by1\ifnum\count255<#1\repeat\rightarrow}\;}
\newcommand{\Right}[2]{\overset{#2}{\longright#1}}
\newcommand{\RIGHT}[3]{\mathrel{\mathop{\kern0pt\longright#1}
	\limits^{#2}_{#3}}}

\newcommand{\LEFT}[3]{\mathrel{\mathop{\kern0pt\longleft#1}\limits^{#2}_{#3}}
}
\newcommand{\longleftright}[1]{\;{\leftarrow\mathrel{\mkern-6mu}%
    \count255=0\loop\relbar\mathrel{\mkern-6mu}% 
    \advance\count255 by1\ifnum\count255<#1\repeat\rightarrow}\;} 
 
\newcommand{\onto}[1]{\;{\count255=0 \loop \relbar\joinrel
    \advance\count255 by1
    \ifnum\count255<#1 \repeat \twoheadrightarrow}\;}

\newcommand{\RLEFT}[3]{\mathrel{%
   \mathop{\vcenter{\baselineskip=0pt\hbox{$\kern0pt\longright#1$}%
   \hbox{$\kern0pt\longleft#1$}}}\limits^{#2}_{#3}}}
\newcommand{\RRIGHT}[3]{\mathrel{%
   \mathop{\vcenter{\baselineskip=0pt\hbox{$\kern0pt\longright#1$}%
   \hbox{$\kern0pt\longright#1$}}}\limits^{#2}_{#3}}}

\numberwithin{table}{section}

\renewenvironment{enumerate}[1][]
{\begin{enumerat}[#1]\setlength{\itemsep}{6pt}\setlength{\topsep}{20pt}
\setlength{\partopsep}{20pt}\setlength{\parsep}{20pt}}
{\end{enumerat}}

\renewenvironment{itemize}
{\begin{itemiz}\setlength{\itemsep}{6pt}\setlength{\itemindent}{-20pt}}
{\end{itemiz}}

\newenvironment{enuma}{\begin{enumerate}[{\rm(a) }]}{\end{enumerate}}
\newtheorem{Thm}{Theorem}[section]
\newtheorem{Prop}[Thm]{Proposition}

\theoremstyle{definition}

\newtheorem{Defi}[Thm]{Definition}

\theoremstyle{remark}

\title{Correction to: Reductions to simple fusion systems}

\author{Bob Oliver}
\address{Universit\'e Sorbonne Paris Nord, LAGA, UMR 7539 du CNRS, 
99, Av. J.-B. Cl\'ement, 93430 Villetaneuse, France.}
\email{bobol@math.univ-paris13.fr}
\thanks{B. Oliver is partially supported by UMR 7539 of the CNRS}


\begin{document}

\begin{abstract} 
We fill in a gap in the proof of the main theorem in our earlier paper 
\cite{O-simp}. At the same time, we prove a slightly stronger version of 
the theorem needed for another paper.
\end{abstract}

\maketitle

The main theorem in our earlier paper \cite{O-simp} stated (very roughly) 
that if $\cale\le\calf$ are saturated fusion systems such that $\cale$ is 
normal in $\calf$ and satisfies certain additional conditions, then there 
is a sequence of saturated fusion subsystems 
$\cale=\calf_0\nsg\calf_1\nsg\cdots\nsg\calf_m=\calf$, each normal in the 
following system and normal in $\calf$, such that $\calf_i$ has $p$-power 
index or index prime to $p$ in $\calf_{i+1}$ for each $i$. We refer to 
\cite[Theorem 2.3]{O-simp}, or to Theorems \ref{t:p-solv} and \ref{solv} 
below, for the precise statement. 

The theorem was proven by an inductive argument, where we assume that 
$\calf_{i+1}$ has already been constructed with certain properties before 
constructing $\calf_i$. This inductive argument requires that $\cale$ be 
normal in $\calf_i$ for each $i$, a property that was not justified in 
\cite{O-simp}. The missing details are not hard to fill in, but we think 
it's best to do so formally, especially since the theorem has been applied by 
various people, either directly as in \cite{HL}, or indirectly via Lemma 
2.22 in \cite{AO}. 

Most of the notation and terminology used in \cite{O-simp} will be assumed 
here; we refer to that paper for their definitions. As one exception, since 
the details of the definition of normal fusion subsystems play an important 
role here, we begin by recalling them. 

\begin{Defi} \label{d:E<|F}
Let $\cale\le\calf$ be saturated fusion systems over finite $p$-groups 
$T\le S$. The subsystem $\cale$ is \emph{weakly normal} in $\calf$ if 
\begin{itemize} 

\item $T$ is strongly closed in $\calf$ (in particular, $T\nsg S$), and 

\item (strong invariance condition) for each $P\le Q\le T$, each 
$\varphi\in\Hom_\cale(P,Q)$, and each $\psi\in\homf(Q,T)$, 
$\psi\varphi(\psi|_P)^{-1}\in\Hom_\cale(\psi(P),\psi(Q))$.

\end{itemize}
The subsystem $\cale$ is \emph{normal} in $\calf$ ($\cale\nsg\calf$) if it 
is weakly normal and 
\begin{itemize} 

\item (extension condition) each $\alpha\in\Aut_\cale(T)$ extends to some 
$\4\alpha\in\autf(TC_S(T))$ such that $[\4\alpha,C_S(T)]\le Z(T)$.

\end{itemize}
\end{Defi}

This is different than the definition of a normal fusion subsystem used in 
\cite{O-simp}, but the two definitions are equivalent by \cite[Proposition 
I.6.4]{AKO}. This one has the advantage that it simplifies the 
proof of point (a) in the following proposition.

\begin{Prop} \label{l:E<|F0}
Let $\cale\le\calf_0\le\calf$ be saturated fusion systems over $T\le S_0\le 
S$. Then the following hold.
\begin{enuma} 
\item If $\cale$ is weakly normal in $\calf$, then $\cale$ is weakly normal 
in $\calf_0$.
\item If $\cale\nsg\calf$ and $\cale=O^{p'}(\cale)$, then 
$\cale\nsg\calf_0$.\
\item If $\cale\nsg\calf$ and $\calf_0$ has $p$-power index in $\calf$ 
\cite[Definition I.7.3]{AKO}, then $\cale\nsg\calf_0$. 
\end{enuma}
\end{Prop}

\begin{proof} If $T$ is strongly closed in $\calf$, then it is also strongly 
closed in $\calf_0$. If the strong invariance condition holds for 
$\cale\le\calf$, then it also holds for $\cale\le\calf_0$. This proves (a). 

Point (b) was shown by David Craven in \cite[Corollary 8.19]{Craven}.

Under the hypotheses of (c), $\cale$ is weakly normal in $\calf_0$ by (a), 
and it remains to prove that each $\alpha\in\Aut_\cale(T)$ extends to 
$\4\alpha\in\Aut_{\calf_0}(TC_{S_0}(T))$ such that 
$[\4\alpha,C_{S_0}(T)]\le Z(T)$. This clearly holds for $\alpha\in\Inn(T)$, 
and since $\Inn(T)\in\sylp{\Aut_\cale(T)}$, it suffices to show it for 
$\alpha\in\Aut_\cale(T)$ of order prime to $p$. For such $\alpha$, since 
$\cale\nsg\calf$, there is $\5\alpha\in\autf(TC_S(T))$ such that 
$\5\alpha|_T=\alpha$ and $[\5\alpha,C_S(T)]\le Z(T)$, and $\5\alpha$ 
restricts to $\4\alpha\in\autf(TC_{S_0}(T))$ since $[\5\alpha,C_S(T)]\le 
T$. Upon replacing $\4\alpha$ by $\4\alpha^k$ for some appropriate 
$k\equiv1$ (mod $|\alpha|$), we can arrange that $\4\alpha$ has order prime 
to $p$ (and still $\4\alpha|_T=\alpha$ and $[\4\alpha,C_{S_0}(T)]\le 
Z(T)$). But then $\4\alpha\in 
O^p(\autf(TC_{S_0}(T)))\le\Aut_{\calf_0}(TC_{S_0}(T))$ since $\calf_0$ has 
$p$-power index in $\calf$, proving the extension condition for 
$\cale\le\calf_0$. \end{proof}

We refer to \cite[Example 8.18]{Craven} for an example of saturated fusion 
systems $\cale\le\calf\le\5\calf$ where $\cale\nsg\5\calf$, and $\cale$ is 
weakly normal but not normal in $\calf$.

\begin{Prop} [{Compare with \cite[Proposition 1.8]{O-simp}}] \label{O^p'(F)}
Let $\cale\nsg\calf$ be saturated fusion systems over finite $p$-groups $T\nsg S$. Let 
$\chi_0\:\autf(T)\too\Delta$ be a surjective homomorphism, for some 
$\Delta\ne1$ of order prime to $p$, such that 
$\Aut_\cale(T)\le\Ker(\chi_0)$. Then there is a unique proper 
normal subsystem $\calf_0\nsg\calf$ over $S$ such that 
	\beqq \Aut_{\calf_0}(S)=\bigl\{\alpha\in\autf(S)\,\big|\,
	\alpha|_T\in\Ker(\chi_0) \bigr\}, \label{e:AutF0(S)} \eeqq
and $\Aut_{\calf_0}(T)=\Ker(\chi)$ and $\cale\nsg\calf_0$. 

%%In particular, $O^{p'}(\calf)\le\calf_0<\calf$.
\end{Prop}

\begin{proof} This was shown in \cite[Proposition 1.8]{O-simp}, except for 
the statements that $\Aut_{\calf_0}(T)=\Ker(\chi)$ and $\cale$ is normal in 
$\calf_0$. Since $\cale$ is weakly normal in $\calf_0$ by Proposition 
\ref{l:E<|F0}(a), normality will follow one we have checked the extension 
condition. 

\iffalse
Fix $\alpha\in\Aut_\cale(T)$. Since $\cale\nsg\calf$, there is 
$\4\alpha\in\autf(TC_S(T))$ that extends $\alpha$ and such that 
$[\4\alpha,C_S(T)]\le Z(T)$. We must show that 
$\4\alpha\in\Aut_{\calf_0}(TC_S(T))$. 
\fi

Set $\calh^*=\{P\in\calf^c\,|\,P\cap T\in\cale^c\}$. Thus 
$TC_S(T)\in\calh^*$. In the proof of \cite[Proposition 1.8]{O-simp}, we 
construct a map 
	\[ \5\chi\: \Mor(\calf|_{\calh^*}) \Right4{} \Delta \]
with the property that for each $P\in\calh^*$ such that $P\ge T$, and each 
$\beta\in\autf(P)$, we have $\5\chi(\beta)=\chi_0(\beta|_T)$ and 
$\Aut_{\calf_0}(P)=\autf(P)\cap\5\chi^{-1}(1)$. (Note that 
$\beta|_T\in\autf(T)$ since $T$ is strongly closed in $\calf$.) So 
\eqref{e:AutF0(S)} holds, and 
	\beqq \Aut_{\calf_0}(TC_S(T)) = \{ \alpha\in\autf(TC_S(T)) \,|\, 
	\alpha|_T\in\Ker(\chi_0) \}. \label{e:xx} \eeqq
Since each $\beta\in\Aut_{\calf_0}(T)$ extends to some 
$\4\beta\in\Aut_{\calf_0}(TC_S(T))$ by the extension axiom 
\cite[Proposition I.2.5]{AKO} applied to $\calf_0$, \eqref{e:xx} shows that 
$\Aut_{\calf_0}(T)\le\Ker(\chi_0)$. Similarly, each 
$\gamma\in\Ker(\chi_0)$ extends to some 
$\4\gamma\in\Aut_{\calf}(TC_S(T))$ by the extension axiom for $\calf$, 
and $\4\gamma,\gamma\in\Mor(\calf_0)$ by \eqref{e:xx} again. Thus 
$\Aut_{\calf_0}(T)=\Ker(\chi_0)$.

For each $\alpha\in\Aut_\cale(T)$, the extension condition for 
$\cale\nsg\calf$ implies that there exists 
$\4\alpha\in\autf(TC_S(T))$ extending $\alpha$ and with 
$[\4\alpha,C_S(T)]\le Z(T)$. Then $\4\alpha\in\Aut_{\calf_0}(TC_S(T))$ by 
\eqref{e:xx} and since $\4\alpha|_T=\alpha\in\Aut_\cale(T)$. So the 
extension condition holds for $\cale\le\calf_0$, proving that 
$\cale\nsg\calf_0$. 
\end{proof}

\begin{Defi} 
Let $\cale\nsg\calf$ be saturated fusion systems over finite $p$-groups 
$T\nsg S$, and define 
	\[ C_S(\cale) = \{ x\in S\,|\, C_\calf(x)\ge \cale \}. \]
We say that $\cale$ is \emph{centric} in $\calf$ if $C_S(\cale)\subseteq T$.
\end{Defi}

By a theorem of Aschbacher (see Notation 6.1 and (6.7.1) in \cite{A-gfit}), 
for each such $\cale\nsg\calf$, $C_S(\cale)$ is a subgroup of $S$, and 
$C_\calf(C_S(\cale))$ contains $\cale$. Thus each morphism in $\cale$ 
extends to a morphism in $\calf$ between subgroups containing $C_S(\cale)$ 
that is the identity on $C_S(\cale)$.

For each saturated fusion system $\calf$ over a finite $p$-group $S$, we 
set 
	\[ \Aut(\calf) = \{\beta\in\Aut(S) \,|\, \9\beta\calf=\calf \}: \]
the group of ``fusion preserving'' automorphisms of $S$. (This group was 
denoted $\Aut(S,\calf)$ in \cite{O-simp}.) For $\beta\in\Aut(\calf)$, let 
$c_\beta$ be the automorphism of the category $\calf$ that sends $P$ to 
$\beta(P)$ and sends $\varphi\in\homf(P,Q)$ to 
$\beta\varphi\beta^{-1}\in\homf(\beta(P),\beta(Q))$.

The next theorem contains most of Theorems 1.14 and 2.3 in \cite{O-simp}, 
together with some additional information about automorphisms of the 
systems. 

%%[Compare with Theorems 1.14 and 2.3 in \cite{O-simp}]

\begin{Thm} \label{t:p-solv}
Let $\cale\nsg\calf$ be saturated fusion systems over finite $p$-groups 
$T\nsg S$. Assume that $\autf(T)/\Aut_\cale(T)$ is $p$-solvable 
(equivalently, that $\outf(T)$ is $p$-solvable). 
\begin{enuma} 

\item In all cases, there is a sequence 
	\beqq 
	\parbox{\shorter}{$\calf_0\le\calf_1\le\calf_2\le\cdots\le 
	\calf_m=\calf$ of saturated fusion subsystems (for some $m\ge0$) such 
	that for each $0\le i< m$, 
	$\calf_i$ is normal of $p$-power index or index prime to $p$ in 
	$\calf_{i+1}$ and $\cale\nsg\calf_i\nsg\calf$; 
	}
	\label{e:yy} \eeqq
and such that $\calf_0$ is a fusion system over $TC_S(T)$ and 
$\Aut_{\calf_0}(T)=\Aut_\cale(T)$.

\item If $\cale$ is centric in $\calf$, then there is a sequence of 
subsystems satisfying \eqref{e:yy} such that $\calf_0=\cale$. 

\end{enuma}
In either case, the subsystems can be chosen so that for each 
$1\le j\le m$, and each $\beta\in\Aut(\calf_j)$ with 
$c_\beta(\cale)=\cale$, we have $c_\beta(\calf_i)=\calf_i$ for all $0\le 
i<j$. 
\end{Thm}

%%\mynote{Remove all mention of $\calf^\infty$?}

\begin{proof} We outline here the proof as given in \cite{O-simp}: enough 
to explain how Propositions \ref{l:E<|F0}(c) and \ref{O^p'(F)} are used to 
prove that $\cale\nsg\calf_i$ for each $i$, and explain why the last 
statement is true. We refer frequently to the following transitivity result 
for normality (see \cite[7.4]{A-gfit}):
	\beqq  \parbox{\shorter}{If $\calf_2\nsg\calf_1\nsg\calf$ are 
	saturated fusion systems over finite $p$-groups $S_2\nsg S_1\nsg S$ 
	such that $c_\alpha(\calf_2)=\calf_2$ for each 
	$\alpha\in\Aut_\calf(S_1)$, then $\calf_2\nsg\calf$.}
	\label{e:zz} \eeqq

\smallskip

\noindent\textbf{(a) } Set $G=\autf(T)$ and $G_0=\Aut_\cale(T)$. Since 
$G/G_0$ is $p$-solvable, there are subgroups $G_0\nsg G_1\nsg \cdots\nsg 
G_m=G$ (some $m\ge0$) such that for each $0\le i<m$, either 
$G_i=O^p(G_{i+1})G_0$ (hence $G_{i+1}/G_{i}$ is a $p$-group), or 
$G_i=O^{p'}(G_{i+1})G_0$ (hence $G_{i+1}/G_i$ has order prime to $p$). In 
particular, the $G_i$ are all normal in $G$ since $G_0$ is. For each $i$, 
set $S_i=N_S^{G_i}(T)\defeq\{x\in S\,|\,c_x\in G_i\}$. Thus $S_i\nsg S$ and 
$\Aut_{S_i}(T)=G_i\cap\Aut_S(T)\in\sylp{G_i}$ for each $i$, $S_m=S$, and 
$S_0=TC_S(T)$. We will construct successively subsystems 
$\calf=\calf_m\ge\calf_{m-1}\ge\cdots\ge\calf_0$ in $\calf$ such that for 
each $0\le i\le m-1$, $\calf_i$ is a fusion system over $S_i$, 
$\Aut_{\calf_i}(T)=G_i$, and the conditions on $\calf_i$ in \eqref{e:yy} 
all hold.

Assume, for some $0\le i< m$, that $\calf_{i+1}\nsg\calf$ has been 
constructed satisfying these conditions. Thus $\Aut_{\calf_{i+1}}(T)=G_{i+1}$. 
If $G_{i+1}/G_{i}$ has order prime to $p$, then by Proposition \ref{O^p'(F)}, 
applied with $G_{i+1}/G_{i}$ in the role of $\Delta$, there is a unique 
saturated subsystem $\calf_{i}\nsg\calf_{i+1}$ of index prime to $p$ over 
$S_{i+1}=S_{i}$ such that 
$\Aut_{\calf_{i}}(S_{i+1})=\{\alpha\in\Aut_{\calf_{i+1}}(S_{i+1})\,|\,\alpha|_T\in 
G_{i}\}$, and also $\cale\nsg\calf_{i}$ and 
$\Aut_{\calf_{i}}(T)=G_{i}$. 

If $G_{i+1}/G_{i}$ is a $p$-group, then the argument in the proof of 
\cite[Theorem 1.14]{O-simp} shows that there is a unique 
$\calf_{i}\nsg\calf_{i+1}$ over $S_{i}$ of $p$-power index, and 
$\cale\le\calf_{i}$ by \cite[Proposition 1.11]{O-simp}. Then 
$\cale\nsg\calf_{i}$ by Proposition \ref{l:E<|F0}(c). Since 
$\Aut_{\calf_{i}}(T)$ has $p$-power index in $\Aut_{\calf_{i+1}}(T)$, we 
have 
	\[ \Aut_{\calf_{i}}(T) = 
	O^p(\Aut_{\calf_{i+1}}(T))\Aut_{S_{i}}(T) = 
	O^p(G_{i+1})G_{i}=G_{i}: \]
where the first equality holds since 
$\Aut_{S_{i}}(T)\in\sylp{\Aut_{\calf_{i}}(T)}$, and the second since 
$G_{i+1}=\Aut_{\calf_{i+1}}(T)$ and $\Aut_{S_{i}}(T)\in\sylp{G_{i}}$. Thus 
$\Aut_{\calf_{i}}(T)=G_{i}$.

\iffalse
In all cases, $\calf_{i}\nsg\calf$ by \eqref{e:zz} and the uniqueness of 
the choices involved (and since $G_{i-1}\nsg G$). 
\fi

\iffalse
This finishes the proof of point (a), except for the last statement. To see 
this, choose $1\le j\le m$ and 
$\beta\in\Aut(\calf_j)\le\Aut(S_j)$ such that $c_\beta(\cale)=\cale$. Then 
$\beta(T)=T$ and $\9\beta G_0=G_0$, and hence $\9\beta G_j=G_j$ for each 
$0\le i<j$ by definition of the $G_i$. So by the uniqueness of the choices 
(depending only on the $G_i$), we get that $c_\beta(\calf_i)=\calf_i$ for 
each $0\le i<j$. In particular, when $j=m$, this says that 
$c_\beta(\calf_i)=\calf_i$ for each $\beta\in\Aut_\calf(S)\le\Aut(\calf)$ 
and each $i<m$, and together with \eqref{e:zz}, it implies that 
$\calf_i\nsg\calf$ for each $i$.
\fi

\smallskip

\noindent\textbf{(b) } By (a) and \eqref{e:zz}, it suffices to prove this when 
$S=TC_S(T)$. By \cite[Corollary 2.2]{O-simp}, $S/T=TC_S(T)/T$ is abelian. 

Set $\calh=\{P\le S\,|\,P\ge C_S(T)\}$, 
and let $\calf^*\subseteq\calf$ be the full subcategory with 
$\Ob(\calf^*)=\calh$. Define 
	\[ \chi\: \Mor(\calf^*) \Right5{} \Aut_{\calf/T}(S/T) \]
by sending $\varphi\in\Hom_{\calf}(P,Q)$ to the induced automorphism of 
	\[ S/T = PT/T = QT/T \cong P/(P\cap T) \cong Q/(Q\cap T). \] 
Here, $PT=QT=S$ since 
$P,Q\in\calh$ and $S=TC_S(T)$, and $\varphi(P\cap T)\le Q\cap T$ 
since $T$ is strongly closed. Thus each $\varphi\in\Mor(\calf^*)$ factors 
through some $\4\varphi\in\Aut(S/T)$. In the notation of Craven 
\cite[Definition 5.5]{Craven}, $\4\varphi\in\Aut_{\4\calf_T}(S/T)$, and 
so $\4\varphi\in\Aut_{\calf/T}(S/T)$ by \cite[Theorem 5.14]{Craven}. 
See also \cite[Theorem 12.5]{A-gfit} for a different proof that 
$\4\varphi\in\Mor(\calf/T)$.

We now apply \cite[Lemma 1.6]{O-simp}, whose hypotheses (i)--(v) are shown 
to hold in the proof of \cite[Theorem 2.3]{O-simp}. By that lemma, 
$\calf_2\defeq\gen{\chi^{-1}(1)}$ is a saturated fusion subsystem over $S$ 
normal of index prime to $p$ in $\calf$ such that 
$\Aut_{\calf_2}(S)=\Ker(\chi|_{\Aut_{\calf}(S)})$. 

\iffalse
Of the conditions (i)--(v) in that 
lemma, the only one that is not immediate is (ii). This says that each 
$P\in\calf^c\sminus\calh$ is $\calf$-conjugate to some $P^*$ such that 
$\Out_{S}(P^*)\cap O_p(\Out_{\calf}(P^*))\ne1$, and it holds since for such 
$P$ and $x\in N_{C_{S}(T)}(P)\sminus P$, $c_x\in O_p(\Out_{\calf}(P))$ 
since it induces the identity on $P\cap T$ and on $P/(P\cap T)$. 
\fi

By Proposition \ref{l:E<|F0}(a), $\cale$ is weakly normal in $\calf_2$. If 
$\alpha\in\Aut_\cale(T)$, then since $\cale\nsg\calf$, $\alpha$ extends to 
$\4\alpha\in\autf(S)$ such that $[\4\alpha,C_S(T)]\le Z(T)$. Since 
$S=TC_S(T)$, this implies that $\chi(\4\alpha)=1$, and hence that 
$\4\alpha\in\Aut_{\calf_2}(S)$. So the extension condition holds, and 
$\cale\nsg\calf_2$. 

The construction of $\calf_1\nsg\calf_2$ of $p$-power index such that 
$\cale\nsg\calf_1$ and has index prime to $p$ follows from exactly 
the same argument as used in \cite{O-simp}, except that $\cale$ is normal 
in $\calf_1$ by Proposition \ref{l:E<|F0}(c). 

\iffalse
The last statement (invariance under automorphisms) follows, as in (a), 
from the uniqueness of the choices of subsystems at each stage. Also, 
$\calf_1\nsg\calf$ by that and \eqref{e:zz}. 
\fi

\smallskip

\noindent\textbf{(a,b) } It remains to prove the last statement (invariance 
under automorphisms), and show that $\calf_i\nsg\calf$ for all $i$ (not 
only for $i=m-1$). To see this, choose $1\le j\le m$ and 
$\beta\in\Aut(\calf_j)\le\Aut(S_j)$ such that $c_\beta(\cale)=\cale$. Then 
$\beta(T)=T$ and $\9\beta\Aut_\cale(T)=\Aut_\cale(T)$. In (b), we have 
$c_\beta(\calf_i)=\calf_i$ for $0\le i<j$ by the uniqueness of choices of 
subsystems at each stage. In (a), 
$\9\beta\Aut_{\calf_i}(T)=\Aut_{\calf_i}(T)$ for each $0\le i<j$ by 
construction of $G_i=\Aut_{\calf_i}(T)$, and hence by the uniqueness of the 
choices (depending only on the $G_i$), we have 
$c_\beta(\calf_i)=\calf_i$. 

\iffalse
in (b) by 
construction of the $\calf_i$. By the uniqueness of the choices (depending 
in (a) only on the $G_i$), we now get that $c_\beta(\calf_i)=\calf_i$ for 
each $0\le i<j$. 
\fi

\iffalse
So $\9\beta\Aut_{\calf_i}(T)=\Aut_{\calf_i}(T)$ for each $0\le i<j$: this 
holds in (a) by construction of $G_i=\Aut_{\calf_i}(T)$, and in (b) by 
construction of the $\calf_i$. By the uniqueness of the choices (depending 
in (a) only on the $G_i$), we now get that $c_\beta(\calf_i)=\calf_i$ for 
each $0\le i<j$. 
\fi

In particular, if $0\le i<m$ is such that $\calf_{i+1}\nsg\calf$, this says 
that $c_\beta(\calf_i)=\calf_i$ for each 
$\beta\in\Aut_\calf(S_{i+1})\le\Aut(\calf_{i+1})$, and together with 
\eqref{e:zz}, it implies that $\calf_i\nsg\calf$. It now follows 
inductively that $\calf_i\nsg\calf$ for each $i$.
\end{proof}

For each saturated fusion system $\calf$ over a finite $p$-group $S$, set 
$\calf^\infty=\calf_0$ for any sequence 
$\calf_0\nsg\calf_1\nsg\cdots\nsg\calf_m=\calf$ of saturated subsystems 
such that $O^p(\calf_0)=O^{p'}(\calf_0)=\calf_0$, and such that 
$\calf_i\nsg\calf$ and $\calf_i$ has index prime to $p$ or $p$-power index 
in $\calf_{i+1}$ for each $0\le i<m$. By \cite[Lemma 1.13]{O-simp}, 
$\calf^\infty$ is independent of the choice of the $\calf_i$.

\iffalse
Finally, when $\calf$ is a saturated fusion system over a finite $p$-group 
$S$, we set 
$\Out(\calf)=\{\beta\in\Aut(S)\,|\,\9\beta\calf=\calf\}/\Aut_\calf(S)$. 
Note that this group is denoted $\Out(S,\calf)$ in \cite{O-simp}.
\fi

\begin{Thm}[{\cite[Theorem 2.3]{O-simp}}] \label{solv}
Let $\cale\nsg\calf$ be saturated fusion systems over finite $p$-groups 
$T\nsg S$ such that $\cale$ is centric in $\calf$. Assume either
\begin{enuma} 
\item $\autf(T)/\Aut_\cale(T)$ is $p$-solvable; or 
\item $\Out(\cale)\defeq\Aut(\cale)/\Aut_\cale(T)$ is $p$-solvable.
\end{enuma}
Then $\calf^\infty=\cale^\infty$. 
\end{Thm}

\begin{proof} Since $\autf(T)\le\Aut(\cale)$ (since $\cale\nsg\calf$), 
we have $\autf(T)/\Aut_\cale(T)\le\Out(\cale)$. So (b) 
implies (a). It thus suffices to prove the theorem when (a) holds, and this 
follows immediately from Theorem \ref{t:p-solv}(b) and \cite[Lemma 
1.13]{O-simp}.
\end{proof}

%%\newpage

\end{document}